\documentclass[reqno,11pt]{amsart}
\usepackage{amsmath,amssymb,latexsym,soul,cite,mathrsfs,accents}
\usepackage{color,enumitem,graphicx}
\usepackage[colorlinks=true,urlcolor=blue,
citecolor=red,linkcolor=blue,linktocpage,pdfpagelabels,
bookmarksnumbered,bookmarksopen]{hyperref}
\usepackage[english]{babel}
\usepackage[left=2.6cm,right=2.6cm,top=2.9cm,bottom=2.9cm]{geometry}
\usepackage{tensor}
\usepackage[abbrev]{amsrefs}
\setenumerate{label=(\roman*)}

\DeclareMathOperator{\Aut}{Aut}

\DeclareMathOperator{\Vol}{Vol}

\DeclareMathOperator{\Conf}{Conf}

\newcommand{\Diff}{\mathrm{Diff}}

\newcommand{\cg}{\widetilde{g}}

\newcommand{\cM}{\widetilde{M}}

\newcommand{\cpi}{\widetilde{\pi}}

\newcommand{\bN}{\mathbb{N}}

\newcommand{\bR}{\mathbb{R}}

\def\sideremark#1{\ifvmode\leavevmode\fi\vadjust{\vbox to0pt{\vss
 \hbox to 0pt{\hskip\hsize\hskip1em
 \vbox{\hsize3cm\tiny\raggedright\pretolerance10000
 \noindent #1\hfill}\hss}\vbox to8pt{\vfil}\vss}}}

\newcommand{\suchthat}{\mathrel{}:\mathrel{}}

\newtheorem{theorem}{Theorem}[section]
\newtheorem{proposition}[theorem]{Proposition}
\newtheorem{lemma}[theorem]{Lemma}

\newtheorem{theoremletter}{Theorem}
\newtheorem{propositionletter}{Proposition}
\newtheorem{definitionletter}{Definition}

\theoremstyle{definition}
\newtheorem{definition}[theorem]{Definition}

\newcommand{\innerthmname}{}

\theoremstyle{definition}

\makeatletter
\def\namedlabel#1#2{\begingroup
	#2%
	\def\@currentlabel{#2}%
	\phantomsection\label{#1}\endgroup
}

\def\XXint#1#2#3{{\setbox0=\hbox{$#1{#2#3}{\int}$ }
		\vcenter{\hbox{$#2#3$ }}\kern-.6\wd0}}

\newcommand*\owedge{\mathpalette\@owedge\relax}
\newcommand*\@owedge[1]{%
	\mathbin{%
		\ooalign{%
			$#1\m@th\bigcirc$\cr
			\hidewidth$#1\m@th\wedge$\hidewidth\cr
		}%
	}%
}
\makeatother

\newcommand{\ud}{\mathrm{d}}

\theoremstyle{remark}
\newtheorem{remark}[theorem]{Remark}

\numberwithin{equation}{section}

\title[Nonhomothetic complete periodic metrics with constant scalar curvature]{Nonhomothetic complete periodic metrics with constant scalar curvature}  

\author[J.H. Andrade]{Jo\~{a}o H. Andrade}
\author[J.S. Case]{Jeffrey S. Case}
\author[P. Piccione]{Paolo Piccione}
\author[J. Wei]{Juncheng Wei}

\address[J.H. Andrade]{
	Institute of Mathematics and Statistics,
	University of S\~ao Paulo
	\newline\indent 
	S\~ao Paulo, SP 05508-090, Brazil}
\email{\href{mailto:andradejh@ime.usp.br}{andradejh@ime.usp.br}}

\address[J.S. Case]{
	Department of Mathematics \\
	The Pennsylvania State University
    \newline\indent
    University Park, PA 16802, USA}
\email{\href{mailto:jscase@psu.edu}{jscase@psu.edu}}

\address[P. Piccione]{
    Department of Mathematics, 
    School of Sciences, Great Bay University
    \newline\indent 
    523000, Dongguan-GD, People’s Republic of China
    \newline\indent 
    and
    \newline\indent
    School of Mathematical Sciences, Zhejiang Normal University
    \newline\indent 
    321004, Jinhua-ZJ, People’s Republic of China
    \newline\indent 
    and
    \newline\indent
    (Permanent address) Institute of Mathematics and Statistics,	University of S\~ao Paulo
    \newline\indent 
    05508-090, S\~ao Paulo-SP, Brazil}
    \email{\href{mailto:piccione@ime.usp.br}{paolo.piccione@usp.br}}

\address[J. Wei]{
	Department of Mathematics,
	The Chinese University of Hong Kong
	\newline\indent 
    Room 201, Lady Shaw Building,
	Shatin-NT, Hong Kong}
\email{\href{mailto:wei@math.cuhk.edu.hk}{wei@math.cuhk.edu.hk}}

\subjclass[2020]{Primary 58J55; Secondary 35J60, 35B09, 53C21}
\keywords{Singular Yamabe problem, Constant Scalar Curvature, Bifurcation}

\begin{document}

\begin{abstract}
We show that there are infinitely many pairwise nonhomothetic, complete, periodic metrics with constant scalar curvature that are conformal to the round metric on $S^n\setminus S^k$, where $k < \frac{n-2}{2}$. These metrics are obtained by pulling back Yamabe metrics defined on products of $S^{n-k-1}$ and compact hyperbolic $(k+1)$-manifolds. 
Our main result proves that these solutions are generically distinct up to homothety. 
The core of our argument relies on classical rigidity theorems due to Obata and Ferrand, which characterize the round sphere by its conformal group. 
\end{abstract}

\maketitle

\section{Introduction}
Periodic metrics with constant scalar curvature arise naturally on noncompact manifolds endowed with a cocompact conformal action of groups having infinite profinite completion. This fact has been exploited to obtain multiple solutions of the singular Yamabe problem on noncompact manifolds obtained by removing singular sets from closed manifolds \cites{MR3504948, MR3803113}. Such metrics are obtained by pulling back constant curvature metrics from compact quotients via finite-sheeted topological coverings. Here, multiplicity is meant in the \emph{analytical} sense, {\it i.e.}, the corresponding conformal factors are pairwise independent solutions of the Yamabe equation. This does not, in general, imply that such metrics are nonisometric. 

In this paper, we study the geometric classification of these periodic metrics on $S^n\setminus S^k$ with $k<\frac{n-2}{2}$, focusing on the extent to which they are genuinely distinct up to isometry.
Our main goal is to prove that the periodic metrics arising from inequivalent finite coverings are generically pairwise nonhomothetic. This confirms a conjecture suggested by earlier work of Bettiol, Piccione, and collaborators \cite{MR3803113}*{pag. 600} and \cite{MR3504948}*{pag. 202}, and emphasizes the contrast between analytic and geometric notions of equivalence in the moduli theory of conformal metrics with constant scalar curvature. A central ingredient in our argument is a rigidity theorem of Obata and Ferrand \cites{MR0303464,MR1371767}, which we use to control the size of the conformal automorphism groups of the manifolds' universal covering.

To make this precise, we begin by recalling the variational and conformal structure of the singular Yamabe problem on spheres.

Let $n \geqslant 3$ and $0\leqslant k<\frac{n-2}{2}$. The singular Yamabe problem on the round sphere $(S^n, g_{S^n})$ seeks complete conformally round metrics with constant scalar curvature on $S^n \setminus \Lambda$.
When $\Lambda=S^k$, one can exploit the conformal equivalence
\[
(S^n \setminus S^k, g_{S^n}) \cong (S^{n-k-1} \times {H}^{k+1}, g_{S^{n-k-1}} \oplus g_{{H}^{k+1}}),
\]
to construct solutions as pullbacks of Yamabe metrics on products
$S^{n-k-1} \times \Sigma^{k+1}$ of the sphere and a compact hyperbolic manifold.
We say that a singular Yamabe metric is periodic if it arises in this way.
This perspective connects the analysis of curvature equations on singular spaces to the topology and geometry of their compact quotients.

The first results date back to an ODE analysis by Schoen~\cite{MR994021} of $S^{n-1}$-invariant solutions of the singular Yamabe problem on $S^n\setminus S^0\simeq S^{n-1}\times\mathbb{R}$:

\begin{theoremletter}[{\cite{MR994021}}]
\label{thm:schoen}
Let $k=0$ and $n \geqslant 3$. There exist uncountably many periodic conformally round complete metrics on $S^n \setminus S^0$ with constant scalar curvature.
\end{theoremletter}

As pointed out by Bettiol, Piccione, and Santoro~\cite{MR3504948}, Schoen's result can also be understood via bifurcation theory.
Their approach also produces infinitely many solutions to the singular Yamabe problem when $\Lambda=S^1$:

\begin{theoremletter}[{\cite{MR3504948}}]
\label{thm:BPS}
Let $k=1$ and  $n \geqslant 5$. There exist uncountably many periodic conformally round complete metrics on $S^n \setminus S^1$ with constant scalar curvature.
\end{theoremletter}

The key points here are that $S^n\setminus S^1\simeq S^{n-2}\times H^2$ is the universal cover of $S^{n-2}\times \Sigma^2$ for any compact hyperbolic surface, and there is a rich theory of hyperbolic structures on $\Sigma^2$.
Later, a complementary argument was given by Bettiol and Piccione \cite{MR3803113} based on the geometry of symmetric spaces and the topology of discrete subgroups. 
Their solutions correspond to pullbacks of Yamabe metrics on products $S^{n-k-1} \times \Sigma^{k+1}$, with $\Sigma^{k+1} = {H}^{k+1} / \Gamma$ a compact hyperbolic manifold, where $\Gamma\subset \operatorname{Isom}({H}^{k+1})$ is a cocompact lattice. 

\begin{theoremletter}[{\cite{MR3803113}}]
\label{thm:BP}
Let $k\in\mathbb{Z}_{\geq0}$ and $n \geqslant 2k+3$. 
There exist uncountably many periodic conformally round complete metrics on $S^n \setminus S^k$ with constant scalar curvature.
\end{theoremletter}

Despite the similarities of their statements, there is a subtle distinction between Theorems~\ref{thm:schoen} and~\ref{thm:BPS} on the one hand, and Theorem~\ref{thm:BP} on the other hand, coming from properties of compact hyperbolic manifolds (counting $S^1$ as a hyperbolic manifold).
In the first two cases, there are uncountably many solutions with pairwise distinct periods.
This comes from the facts that there are uncountably many nonisometric metrics on $S^1$ and that the space of hyperbolic structures on a compact surface of genus $g \geq 2$ is a manifold~\cite{Ratcliffe2019}.
In the latter case, it is only known that there are countably many solutions with pairwise distinct periods.
This comes from the fact that the space of compact hyperbolic $n$-manifolds is countable~\cites{Thurston1997,Wang1972}.
Instead, the uncountability comes from the action of the isometry group of hyperbolic space on solutions.

These results naturally lead to the question of how to classify the resulting solutions. To this end, it is useful to distinguish two moduli spaces:
\begin{itemize}
    \item[{\rm (i)}] The \emph{analytic moduli space} $\mathcal{M}_{\mathrm{PDE}}$ consists of conformal factors such that the associated conformal metric has constant scalar curvature, modulo constant rescalings;
    
    \item[{\rm (ii)}] The \emph{geometric moduli space} $\mathcal{M}_{\mathrm{geom}}$ consists of homothety classes of metrics with constant scalar curvature conformal to a fixed background metric $g$.
\end{itemize}
While $\mathcal{M}_{\mathrm{PDE}}$ captures the structure of the solution set to the Yamabe equation as a nonlinear PDE, the geometric moduli space $\mathcal{M}_{\mathrm{geom}}$ encodes which solutions are truly distinct up to homothety.
In the above theorems, the solutions are distinct in $\mathcal{M}_{\rm PDE}$.
We now present the precise definition of these moduli spaces.

Let us start with $(M, g)$ a smooth $n$-dimensional Riemannian manifold with $n \geqslant 3$. 
The classical Yamabe problem seeks to find a metric of constant scalar curvature in the conformal class \[[g] := \left\{u^{\frac{4}{n-2}} g : u>0 \; {\rm and} \; u \in \mathcal{C}^\infty(M) \right\}.\]
We denote by $\mathrm{CSC}(M, [g])$ the set of constant scalar curvature metrics in the conformal class $[g]$.
For any $u \in \mathcal{C}^\infty(M)$ with $u>0$, we define the conformally related metric
\[
g_u := u^{\frac{4}{n-2}} g\in[g].
\]
It is well-known (cf.\ \cite{MR888880}) that the conformal factor $u\in \mathcal{C}^\infty(M)$ scalar curvature of $g_u$ satisfies the nonlinear elliptic PDE
\begin{equation} \label{eq:Yamabe}\tag{\(\mathscr{Y}_n\)}
L_gu := -a_n \Delta_g u + R_g u = R_{g_u} u^{\frac{n+2}{n-2}} \quad {\rm in} \quad M,
\end{equation}
where $a_n=\frac{4(n-1)}{n-2}$ is a dimensional constant and $R_g, \Delta_g$ denotes the scalar curvature and the Laplace--Beltrami operator of $g$, respectively, with the convention $-\Delta_g\geqslant 0$.

The standard formulation of the Yamabe problem seeks positive solutions $u \in \mathcal{C}^\infty(M)$ to~\eqref{eq:Yamabe} such that $g_u$ has constant scalar curvature.
The operator $L_g$ on the left-hand side is the conformal Laplacian. The nonlinearity on the right-hand side has critical growth in the sense of the Sobolev embedding $W^{1,2}(M)\hookrightarrow L^{2^*}(M)$ with $2^*=\frac{2n}{n-2}$.

First, we define the analytic moduli space. 
We emphasize that this is equivalent to the set of solutions to the PDE \eqref{eq:Yamabe} modulo constant rescalings. 
\begin{definitionletter}
{\rm The \emph{analytic moduli space} of constant scalar curvature metrics in $[g]$ is the quotient
\[
\mathcal{M}_{\mathrm{PDE}}(M, [g]) := \mathrm{CSC}(M, [g]) \big/\mathbb{R}_+ \simeq\{u\in C^\infty(M) : u>0 \; {\rm and} \; u \; {\rm solves} \; \eqref{eq:Yamabe}\}\big/\mathbb{R}_+,
\]
where $\mathbb{R}_{+}$ acts by scaling.}
\end{definitionletter}

Second, we define the geometric moduli space.
This encodes metrics that are genuinely distinct up to homothety.
Here we say that two Riemannian metrics $g_1,g_2$ on $M$ are \emph{homothetic} if there is a diffeomorphism $\Phi \in \Diff(M)$ and a constant $c>0$ such that $\Phi^\ast g_1 = c^2g_2$, otherwise they are called \emph{nonhomothetic}.
\begin{definitionletter}
{\rm The {\it geometric moduli space} of constant scalar curvature metrics in $[g]$ is the quotient
\[
\mathcal{M}_{\mathrm{geom}}(M, [g]) := \mathrm{CSC}(M, [g]) \big/ \sim,
\]
where $g_1\sim g_2$ if $g_1$ and $g_2$ are homothetic.} 
\end{definitionletter}

It is not hard to see that the projection
\[
\pi \colon \mathcal{M}_{\mathrm{PDE}}(M, [g]) \longrightarrow \mathcal{M}_{\mathrm{geom}}(M, [g])
\]
is not necessarily injective: two solutions may yield homothetic metrics via a nontrivial diffeomorphism. The fibers of this map encode the failure of the analytic moduli space to resolve geometric equivalence.

For example, in the classical case of the standard round sphere $(S^n, g_{S^n})$, Obata’s theorem \cite{MR0303464} shows that every solution to the Yamabe problem in the conformal class $[g_{S^n}]$ is isometric to $g_{S^n}$. In that case, $\mathcal{M}_{\mathrm{PDE}}$ is a homogeneous space 
\[
\mathcal{M}_{\mathrm{PDE}}(S^n,[g_{S^n}]) \cong \operatorname{Conf}(S^n) / \operatorname{Isom}(S^n) \cong \frac{SO(n+1,1)_0}{SO(n+1)}.
\]
of positive dimension, whereas $\mathcal{M}_{\mathrm{geom}}=\{[g_{S^n}]\}$ is a single point.

Now, we consider the singular Yamabe problem on $S^n$ with singular set $\Lambda=S^k$ modulo scaling by constants.
As before, the analytic moduli space is characterized as the set of positive smooth solutions to the singular Yamabe equation 
\begin{equation} \label{eq:Yamabesingular}\tag{\(\mathscr{Y}^*_{n,k}\)}
\begin{cases}
-a_n \Delta_g u + R_g u = (n-1)(n-2k-2) u^{\frac{n+2}{n-2}} \quad {\rm in} \quad {S}^n\setminus {S}^k,\\
\lim_{\ud_g(x,{S}^k)\to 0}u(x)=\infty.
\end{cases}
\end{equation}
In other words, one has
\[
\mathcal{M}_{\mathrm{PDE}}({S}^n\setminus {S}^k):=\mathcal{M}_{\mathrm{PDE}}({S}^n\setminus {S}^k, [g_{{S}^n}]):=\{u\in C^\infty_+(S^n\setminus S^k) : u \; {\rm solves} \; \eqref{eq:Yamabesingular}\}.
\]
Our normalization is such that the product metric on $S^{n-k-1} \times H^{k+1} \cong S^n \setminus S^k$ solves~\eqref{eq:Yamabesingular}.

The constructions in Theorems \ref{thm:schoen}, \ref{thm:BPS}, and \ref{thm:BP} provide large families of analytic solutions; {\it i.e.} 
\[
\#\mathcal{M}_{\mathrm{PDE}}({S}^n\setminus {S}^k) \gtrsim \aleph_0.
\]
This raises a natural question: {\it Do these distinct analytic solutions correspond to distinct solutions in $\mathcal{M}_{\rm geom}$?}

The purpose of this note is to give a positive answer. 
Our main result is that the periodic solutions to the singular Yamabe problem obtained via a tower of finite-sheeted topological coverings in \cites{MR3803113,MR3504948} are, after passing to a subtower if necessary, pairwise nonhomothetic. 
In particular, this confirms that the geometric moduli space is infinite, in sharp contrast with the rigidity on the round sphere:
\begin{theorem}
\label{thm:main}
Let $k\in\mathbb{Z}_{\geq0}$ and $n \geqslant 2k+3$. There exist countably many pairwise nonhomothetic periodic conformally round complete metrics on $S^n \setminus S^k$ with constant scalar curvature, {\it i.e.}
\[
\#\mathcal{M}_{\mathrm{Geom}}({S}^n\setminus {S}^k, [g_{\rm rd}])\gtrsim\aleph_0.
\]
\end{theorem}

The core geometric input in our argument is the Ferrand--Obata theorem ~\cites{MR1371767,MR0303464,MR1334876}, which states that the only conformal manifolds for which the conformal group acts nonproperly are $\mathbb{R}^n$ and $S^n$ with their flat conformal structures. 
This allows us to distinguish periodic metrics in the universal cover when it is not $\mathbb{R}^n$.

In addition to its intrinsic interest in the geometry of the Yamabe problem, our result is part of a broader circle of questions concerning the classification of periodic solutions to conformally invariant equations, the geometry and topology of moduli spaces of constant curvature metrics, and the interaction between variational bifurcation theory and geometric topology.
In particular, the papers \cites{MR4251294, MR3803113, MR4695635, arXiv:2310.15798} establish existence and bifurcation results for complete metrics that are constant with respect to other types of curvature, demonstrating how similar nonuniqueness phenomena arise in this setting. 
This note complements the existing works by providing a geometric classification of the periodic solutions obtained via finite coverings, demonstrating that they yield pairwise nonhomothetic metrics.

This note is organized as follows. In Section~\ref{sec:preliminaries}, we collect the basic definitions and results concerning the variational setup for the Yamabe problem, basic definitions about residually finite groups, and the Ferrand--Obata theorem.
These serve as foundational tools in our proof. In Section~\ref{sec:main-proof}, we prove our main result on the nonhomothety of the periodic scalar curvature metrics on the universal covering obtained via finite coverings. 

\section{Preliminaries}
\label{sec:preliminaries}

We gather here the main geometric, analytic, and topological tools that form the foundation of our proof. 

\subsection{Existence theory for the classical Yamabe problem}
We study the classical Yamabe problem, which consists of finding a smooth positive solution to \eqref{eq:Yamabe}.
This geometric PDE may be formulated variationally in terms of the Einstein--Hilbert functional on the conformal class $[g]$.

Let $(M,g)$ be a compact $n$-dimensional Riemannian manifold.
The (volume-normalized) total scalar curvature functional $\mathcal{A}:[g]\to \mathbb{R}$ is given by
\begin{equation}\label{eq:EH-functional}
\mathcal{A}(\bar{g}):={\mathrm{vol}_{\bar{g}}(M)^{\frac{2-n}{n}}}{\int_M R_{\bar{g}} \, d{\rm v}_{\bar{g}}}.
\end{equation}
Note that $\mathcal{A}(c^2\bar{g})=\mathcal{A}(\bar{g})$ is scale invariant for all constants $c>0$.
Given
\[ u\in \mathcal{C}^{\infty}_+(M)=\{u\in \mathcal{C}^{\infty}(M) : u>0\} , \]
we set $\bar{g}_u = u^{\frac{4}{n-2}}g$.
Thus, we may rewrite the total scalar curvature functional as a nonlinear Rayleigh quotient $\mathscr{Q}_g:\mathcal{C}^\infty_+(M)\to \mathbb{R}$ defined as 
\begin{equation}\label{eq:Yamabe-quotient}
\mathscr{Q}_g(u) := \frac{\int_M \left( a_n |\nabla u|^2 + R_g \, u^2 \right)\ud{\rm v}_g}{\left( \int_M u^{\frac{2n}{n-2}}\ud{\rm v}_g \right)^{\frac{n-2}{n}}}.
\end{equation}
The infimum of $\mathscr{Q}_g$ over $\mathcal{C}^\infty_+(M)$ defines the Yamabe constant of the conformal class:
\[
Y(M,[g]) := \inf_{u \in C^\infty_+(M)} \mathscr{Q}_g(u).
\]
The variational theory of the Yamabe problem unfolded through the combined efforts of many mathematicians. 
Yamabe~\cite{MR0125546} initiated the program by proposing a direct minimization argument, but his proof contained an analytic gap concerning compactness. 
This was partially addressed by Trudinger~\cite{MR0240748}, and later advanced by Aubin~\cite{MR0448404}, who introduced key comparison techniques. 
Finally, Schoen~\cite{MR788292} resolved the problem in its full generality by employing the positive mass theorem to handle the case when local geometry is insufficient to obtain a key estimate on the Yamabe constant.

Putting together the result of Yamabe \cite{MR0125546}, Trudinger \cite{MR0240748}, Aubin \cite{MR0448404}, and Schoen \cite{MR788292}, we have the following existence result:

\begin{propositionletter}[\cites{MR0125546, MR0240748, MR0448404, MR788292}]\label{prop:Yamabe-existence}
Let $(M,g)$ be a compact $n$-dimensional Riemannian manifold with $n \geqslant 3$. 
Then, $Y(M,[g])\leqslant Y(S^n,[g_{S^n}])$ and the Yamabe constant $Y(M,[g])$ is achieved by some $u\in\mathcal{C}^\infty_+(M)$. 
Moreover, $Y(M,[g])= Y(S^n,[g_{S^n}])$ if and only if $(M,[g])$ is conformally equivalent to $(S^n,[g_{S^n}])$.
\end{propositionletter}

\begin{proof}
    See \cite{MR0125546}*{Theorem~A}, \cite{MR0240748}*{Theorem~2}, \cite{MR0448404}*{Théorème~11}, and \cite{MR788292}*{Theorem~3}.
\end{proof}

For the standard sphere $(S^n, g_{S^n})$, one has the explicit formula
\[
Y(S^n,[g_{S^n}]) = n(n-1) \omega_n^{2/n}.
\]
This was first proved independently by Aubin \cite{MR0448404} and Talenti \cite{MR0463908}.
Another proof, which relies on the existence of minimizers, uses Obata's classical result \cite{MR0303464} that if a compact Riemannian manifold is conformally Einstein and has constant scalar curvature, then it is Einstein; hence, $\mathcal{M}_{\rm PDE}(M,[g])$ is a point if $(M,g)$ is Einstein.
More generally, if $Y(M,[g])\leqslant 0$, then a maximum principle argument shows that $\mathcal{M}_{\rm geom}(M,[g])$ is a point.
However, uniqueness fails in general for conformal manifolds with a positive Yamabe constant.

\subsection{Yamabe metrics on noncompact manifolds and singular solutions on spheres}
The classical Yamabe problem admits a natural extension to noncompact or singular settings. Given a smooth Riemannian manifold $(M, g)$ of dimension $n \geqslant 3$, one seeks complete conformal metrics $\bar{g}_u \in [g]$ on $M \setminus \Lambda$ with constant scalar curvature, where $\Lambda \subset M$ is a closed subset along which $u$ necessarily blows up.
This leads to the singular Yamabe equation
\begin{equation} \label{eq:singularYamabe} \tag{\(\mathscr{Y}^{*}_{n,\Lambda}\)}
\begin{cases}
a_n\Delta_g u + R_g u = \lambda u^{\frac{n+2}{n-2}} \quad \text{in} \quad M \setminus \Lambda, \\
\lim_{x \to \Lambda} u(x) = \infty ,
\end{cases}
\end{equation}
for some constant $\lambda \in \bR$.
The blow-up condition is necessary for the conformal metric $\bar{g} = u^{\frac{4}{n-2}} g$ to be complete on $M \setminus \Lambda$.

A prototypical case is the singular Yamabe problem on the round sphere $(S^n, g_{S^n})$ with singular set $\Lambda = S^k$, where $0 \leqslant k < \frac{n-2}{2}$ (in this case, the PDE formulation is given by \eqref{eq:Yamabesingular}).
The conformal equivalence
\[
(S^n \setminus S^k, g_{S^n}) \cong (S^{n-k-1} \times H^{k+1}, g_{S^{n-k-1}} \oplus g_{H^{k+1}})
\]
allows for a reduction to the compact setting via compact quotients of the hyperbolic factor. 
This strategy has been effectively implemented in~\cites{MR3504948,MR3803113} to construct periodic solutions descending from products $S^{n-k-1} \times \Sigma^{k+1}$.

Earlier, Schoen~\cite{MR994021} had employed bifurcation theory to produce uncountably many complete metrics with isolated singularities on $S^n \setminus \{p, -p\}$, arising as perturbations of Delaunay-type solutions to an associated ODE. These serve as a local model for necks connecting spherical regions in the conformal metric.
Mazzeo and Pacard developed gluing techniques~\cites{MR1356375, MR1425579, MR1712628} to construct singular solutions with higher-dimensional singular sets. Their analysis shows that, under dimension restrictions on the singular locus $\Lambda$, there exist infinite-dimensional families of complete metrics of constant positive scalar curvature on $M \setminus \Lambda$. These are obtained via perturbative constructions around model solutions on the normal bundle, where the critical exponent becomes subcritical due to dimensional reduction.

These developments highlight the rich structure of the analytic moduli space of singular Yamabe metrics in the noncompact setting and lay the foundation for the classification results discussed in the present work.

\subsection{The Ferrand--Obata theorem}
A fundamental feature of the round conformal class on the sphere and its noncompact models is its exceptional degree of symmetry. 
The Ferrand--Obata theorem characterizes in what way they are exceptional.
To explain this, we begin with some definitions:
\begin{definition}
Let \((M, g)\) be a smooth Riemannian manifold of dimension \(n \geq 2\). The \emph{conformal group} \(\Conf(M, [g])\) consists of all diffeomorphisms \(\varphi \in \Diff(M)\) such that \(\varphi^*g = e^{2u} g\) for some smooth function \(u \in C^\infty(M)\) equipped with the \emph{compact-open topology}, {\it i.e.} a sequence \(\{\varphi_k\}_{k\in\mathbb{N}}\subset\Diff(M)\) converges to \(\varphi\) if for every compact set \(K \subset M\), the maps \(\varphi_k\) and their derivatives converge uniformly to \(\varphi\) and its derivatives on \(K\). 
We say that ${\rm Conf}(M,[g])$ \emph{acts properly} on $M$ if 
\(\{g\in {\rm Conf}(M,[g]) : gK\cap K\neq\emptyset\}\) is relatively compact in the compact-open topology whenever $K\subseteq M$ is compact.
\end{definition}

The Ferrand--Obata theorems assert that, except for $\mathbb{R}^n$ and $S^n$ with their flat structures, the conformal automorphism group of a Riemannian manifold acts properly. 

\begin{propositionletter}[\cites{MR0303464,MR1371767}]
\label{thm:obata-ferrand}
Let \((M^n, g)\) be a complete Riemannian manifold of dimension \(n \geqslant 3\) which is conformally equivalent to the round sphere \((S^n, g_{S^n})\). If the conformal group \(\Conf(M, g)\) does not act properly, then \((M, g)\) is conformal to the round sphere $(S^n,g_{S^n})$ or to flat Euclidean space $(\mathbb{R}^n,\delta)$.
\end{propositionletter}

\begin{proof}
    See \cite{MR1371767}*{Theorem~A${}_1$}.
\end{proof}

As we will see below, Proposition~\ref{thm:obata-ferrand} plays a central role in preventing lifts of nonhomothetic metrics to the universal cover from becoming homothetic. 

\subsection{Residual finiteness and profinite completions}
A key step in our construction relies on the existence of an infinite tower of finite regular coverings of the compact manifold \(M\). 
For this, one needs to understand the normal subgroups of finite index in the fundamental group \( \pi_1(M) \). 
This is governed by the concept of residual finiteness.

First, let us define the notion of profinite completion.
\begin{definition}
Let \( \Gamma \) be a discrete group. The \emph{profinite completion} \( \widehat{\Gamma} \) is the inverse limit
\[
\widehat{\Gamma} := \varprojlim_{\substack{N \trianglelefteq \Gamma\\ \mathrm{[}\Gamma : N\mathrm{]} < \infty}} \Gamma/N,
\]
taken over all finite index normal subgroups of \( \Gamma \). 
It is a compact, totally disconnected topological group.
\end{definition}

Second, we define residual finiteness.
\begin{definition}
Let $\Gamma$ be a discrete group. We say that 
\begin{itemize}
    \item[{\rm (i)}] \( \Gamma \) is \emph{residually finite}, if for every nontrivial element \( \gamma \in \Gamma \), there exists a finite index normal subgroup \( N \trianglelefteq \Gamma \) such that \( \gamma \notin N \). 
    Equivalently, the natural map \( \Gamma \to \widehat{\Gamma} \), where \( \widehat{\Gamma} \) denotes the profinite completion of \( \Gamma \), is injective. 
    \item[{\rm (ii)}] $\Gamma$ has \emph{infinite profinite completion}, if it has infinitely many nontrivial finite index normal subgroups (equivalently, its profinite completion is infinite).
\end{itemize}
\end{definition}

If \( \pi_1(M) \) is infinite and residually finite, then it has infinite profinite completion. 
Hence there is a nested sequence
\[ \dotsm \trianglelefteq \Gamma_j \trianglelefteq \dotsm \trianglelefteq \Gamma_1 \trianglelefteq \Gamma_0 := \pi_1(M) \]
of normal subgroups of finite index at least two.
This gives rise to an infinite tower of finite regular coverings
\[
\dotsb \longrightarrow M_j \longrightarrow \dotsb \longrightarrow M_1 \longrightarrow M_0 := M,
\]
where each \( M_j \) corresponds to the covering associated to \( \Gamma_j \trianglelefteq \pi_1(M) \), and the covering map \( \pi_j \colon M_j \to M_{j-1} \) is regular with deck transformation group \( \Gamma_{j-1}/\Gamma_j \) for any $j\in\mathbb{N}$.

In the next remark, we provide some connections between residually finite groups and the Galois theory of coverings and deck transformations.
\begin{remark}
\label{rmk:galois-coverings}
Let \( \widetilde{M} \to M \) be the universal covering of a connected manifold \( M \), with deck transformation group \( \Gamma = \pi_1(M) \). Then, there exists a Galois-type correspondence between subgroups of \( \Gamma \) and connected covering spaces of \( M \). 
More specifically, one has the following properties:
\begin{itemize}
    \item[{\rm (i)}] Each finite index subgroup \( N \leq \Gamma \) corresponds to a \emph{connected finite-sheeted covering} \( M_N \to M \), unique up to isomorphism.
    
    \item[{\rm (ii)}] If \( N \trianglelefteq \Gamma \) is normal, then the covering \( M_N \to M \) is \emph{regular (also called Galois)}, with deck transformation group \( \Gamma / N \). The group \( \Gamma \) \emph{acts transitively} on the fiber of the covering.
    
    \item[{\rm (iii)}] The inverse system of \emph{finite index} normal subgroups \( N \trianglelefteq \Gamma \) defines the \emph{profinite completion} \( \widehat{\Gamma} = \varprojlim \Gamma / N \). This topological group encodes the totality of finite Galois coverings of \( M \).
\end{itemize}

In this sense, the theory of deck transformations realizes the classical Galois correspondence: open subgroups of \( \widehat{\Gamma} \) correspond to finite-sheeted regular coverings of \( M \), and their quotients give the corresponding deck groups.
Thus the assumption that $\pi_1(M)$ has infinite profinite completion ensures the richness of this correspondence.
In particular, arbitrarily large finite regular coverings exist and may be used to construct Yamabe metrics on increasingly large covers.
\end{remark}

Now, we have an important definition for regular coverings.

\begin{definition}
Let $\pi\colon \widetilde M\to M$ be a finite connected covering and $m\in \mathbb{N}_0$. 
We say that $\pi$ has \emph{degree} $m$ if for every $x\in M$ the fibre $\pi^{-1}(x)$ consists of exactly $m$ points, {\it i.e.} \(|\pi^{-1}(x)\bigr| = m\) for all \(x\in M\).
If in addition the covering is regular (Galois), then its degree equals the index of the corresponding subgroup of the fundamental group, {\it i.e.} \(m :=[\pi_1(M)\colon\pi_1(\widetilde M)]\).    
We denote it by $\deg(\pi)=m$.
\end{definition}

The existence of such towers is guaranteed for a large class of manifolds via the Selberg–Mal'cev lemma, which can be stated as follows:

\begin{propositionletter}[]
\label{lem:selberg-malcev-precise}
If \( \Gamma \subset \mathrm{GL}(m, \mathbb{C}) \) is a finitely generated linear group, then \( \Gamma \) is residually finite.
\end{propositionletter}

\begin{proof}
    See \cite{MR1299730}*{Section~7.6}.
\end{proof}

The last result applies to the fundamental groups of compact locally symmetric spaces of noncompact type, such as compact hyperbolic manifolds. 
Thus, if \( M = S^{n-k-1} \times \Sigma^{k+1}\), where \( \Sigma^{k+1}\subset H^{k+1} \) is a compact hyperbolic manifold with \( k+1 \geqslant 2 \), then \(\pi_1(M) \cong \pi_1(\Sigma) \subset \mathrm{SO}(k+1,1)\)
is a cocompact lattice in a real Lie group, and hence residually finite. Consequently, \( M \) admits an infinite tower of finite regular coverings.
This algebraic input enables the geometric construction in our main results. 
By pulling back the conformal class along these coverings and minimizing the Yamabe functional on each \( M_j \), we obtain a sequence of conformal metrics whose lifts to the universal cover \( \widetilde{M} \cong S^n \setminus S^k \) are pairwise nonhomothetic. 
The residual finiteness of the fundamental group thus bridges the gap between the compact and noncompact settings in our construction.

\section{Proof of the main Result}
\label{sec:main-proof}
In this section, we prove our main result: after passing to a subtower if necessary, the periodic scalar curvature metrics constructed via finite coverings are pairwise nonhomothetic. 
We first establish a general result for compact manifolds whose conformal universal cover is not conformally equivalent to Euclidean space and whose fundamental group has infinite profinite completion. 
We then specialize to the case of the singular Yamabe problem on the punctured sphere.

\subsection{Nonhomothetic lifts in the general setting}
To establish that the conformal metrics arising in the tower are pairwise nonhomothetic, we exploit the interaction between topological coverings and variational problems for scalar curvature. 
The main observation is that, under appropriate assumptions on the fundamental group, one can construct a sequence of finite regular coverings with arbitrarily large volume. Since the scalar curvature is invariant under pullback by coverings, but the minimizers of the normalized Yamabe functional are sensitive to volume growth, this leads to a mechanism for producing conformal metrics with constant scalar curvature that are not homothetic to one another.
This approach hinges on the variational characterization of the Yamabe constant and the existence of minimizers. 
The central ingredient enabling the inductive construction is a topological lemma from \cite{MR3803113}, which ensures the existence of coverings with arbitrarily large volume whenever the fundamental group has infinite profinite completion:
\begin{propositionletter}[\cite{MR3803113}]
\label{bettiol-piccione}
Let \((M,g)\) be a compact $n$-dimensional Riemannian manifold with $n\geqslant 3$. 
Suppose that \(\pi_1(M)\) has infinite profinite completion.  
Then for any \(\mathrm{v} \in \mathbb{R}\) there exists a finite regular covering \(\pi \colon \cM \to M\) such that \(\operatorname{Vol}_{\pi^\ast g}(\widetilde{M}) > \mathrm{v}\).
\end{propositionletter}

\begin{proof}
    See \cite{MR3803113}*{Lemma 3.6}.
\end{proof}

Lifting Yamabe minimizers to the conformal universal cover produces infinitely many representatives with constant scalar curvature for which no two conformal factors are constant multiples of one another.
To conclude that these representatives are in fact nonhomothetic uses the Ferrand--Obata theorem.

To explain this, let us first introduce some notation. 
By an \emph{infinite towering sequence} $\{\pi_j \colon M_j \to M_{j-1}\}_{j \in \bN}$ of finite connected coverings of degree $m_j \geqslant 2$, we mean a tower of nested Riemannian covering maps such that
\[
\dotsb \overset{\pi_{k+1}}{\longrightarrow} M_k \overset{\pi_k}{\longrightarrow} \dotsb \overset{\pi_2}{\longrightarrow} M_1 \overset{\pi_1}{\longrightarrow} M_0 := M \quad {\rm and} \quad m_j=\deg(\pi_j).
\]
For each $j \in \bN$, we set
\(\Pi_j := \pi_1 \circ \dotsb \circ \pi_j \colon M_j \to M \), which is itself a covering of degree
\(\deg(\Pi_j) = \prod_{\ell=1}^j m_\ell= m_1m_2\cdots m_j\).
We denote the universal covering by \(\Pi_\infty \colon {M}_\infty \to M \).
Since $\deg(\pi_j)=m_j\geqslant 2$ for each $j\in\mathbb{N}$, we see that $\Pi_\infty$ has infinite degree.
Hence $M_\infty$ is noncompact.
The main auxiliary result needed in our construction of nonhomothetic constant scalar curvature conformal metrics is as follows:
\begin{proposition}
    \label{construction-lemma}
    Let $(M,g)$ be an $n$-dimensional compact Riemannian manifold with $n \geqslant 3$ and such that $Y(M,[g]) > 0$.
    Suppose that there exists an infinite tower $\{\pi_j \colon M_j \to M_{j-1}\}_{j \in \bN}$ of finite connected coverings of degree $m_j \geqslant 2$ and a sequence $\{g_j\}_{j \in \bN}$ of unit volume Yamabe minimizers on $(M_j,[\Pi_j^\ast g])$ such that for each $j \in \bN$, there exists $\Phi_j \in \Diff(\cM)$ and $c_j \in \mathbb{R}_{>0}$ such that $\Phi_j^\ast \cpi^\ast g = c_j^2 \cpi_j^\ast g_j$, where $\cpi \colon \cM \to M$ and $\cpi_j \colon \cM \to M_j$ are the universal covers of $M$ and $M_j$, respectively.
    Then the universal cover $(\cM,\cpi^\ast g)$ is conformally equivalent to flat Euclidean space $(\mathbb{R}^n,\delta)$.
\end{proposition}

\begin{proof}
Since the Yamabe constant is positive and the statement depends only on the conformal class, we may assume without loss of generality that $R_g = Y(M,[g])$ and ${\rm Vol}_g(M) = 1$.
Also, using that $\lim_{j\to\infty}{\rm deg}(\Pi_j)=\infty$, it follows that the universal cover $\cM$ is noncompact.

Fix $j \in \bN$ and note that $\cpi = \Pi_j \circ \cpi_j$.
Let $F \subset \cM$ be a fundamental domain for $\cpi$ and $\Aut(\cpi), \Aut(\cpi_j)$ be the groups of deck transformations of $\cpi,\cpi_j$, respectively.
Using that $\Aut(\cpi_j)$ has finite index in $\Aut(\cpi)$, there exists $\tau_j \in \Aut(\cpi)$ such that
\[
    {\rm Vol}_{\tau_j^\ast \cpi_j^\ast g_j}(F) = \min \left\{ {\rm Vol}_{\sigma^\ast \cpi_j^\ast g_j}(F) \suchthat \sigma \in \Aut(\cpi) \right\}.
\]
This yields 
\[
{\rm Vol}_{\tau_j^\ast \cpi_j^\ast g_j}(F) \leqslant m_j^{-1}.
\]
    
We choose $\sigma_j \in \Aut(\cpi)$ such that $\Psi_j := \sigma_j \circ \Phi_j \circ \tau_j$ satisfies $\Psi_j(F) \cap F \not= \emptyset$.
Observing that $\bar{g}_j:=\Psi_j^\ast\cpi^\ast g = c_j^2 \tau_j^\ast \cpi_j^\ast g_j$, we compute the total scalar curvature of $\bar{g}_j\rvert_F$ defined as 
\[
    \mathcal{A}(\bar{g}_j) := {{\rm Vol}_{\Psi_j^\ast \cpi^\ast g}(F)^{\frac{2-n}{n}}}{\int_F R_{\Psi_j^\ast \cpi^\ast g} \ud{\rm Vol}_{\Psi_j^\ast \cpi^\ast g}}
\]
in two ways as follows.
    
On the one hand, using $R_g = Y(M,[g])$ and the diffeomorphism invariance of the scalar curvature under pullback, we obtain
\begin{equation}
    \label{eqn:mI-scalar-Psi}
    \mathcal{A}(\bar{g}_j) = Y(M,[g]){\rm Vol}_{\cpi^\ast g}(\Psi_j(F))^{\frac{2}{n}}.
\end{equation}
    
On the other hand, since $\mathcal{A}$ is scale invariant and each $g_j$ minimizes $Y(M_j,[\Pi_j^\ast g])$ with unit volume, we find that
\[
    \mathcal{A}(\bar{g}_j) = Y(M_j,[\Pi_j^\ast g]) {\rm Vol}_{\tau_j^\ast \cpi_j^\ast g_j}(F)^{\frac{2}{n}} \leqslant Y(M_j,[\Pi_j^\ast g]) m_j^{-\frac{2}{n}}.
\]
Thus, $\limsup_{j \to \infty} \mathcal{A}(\bar{g}_j) \leqslant 0$.
Since $Y(M,[g]) > 0$, it follows from \eqref{eqn:mI-scalar-Psi} that ${\rm Vol}_{\cpi^\ast g}(\Psi_j(F)) \to 0$ as $j\to\infty$.
Therefore, the sequence of maps $\{\Psi_j\}_{j\in\mathbb N}\subset \Aut(\cpi)$ cannot be precompact in the compact-open topology.
From this, we conclude that the conformal group of $(\cM, \cpi^\ast g)$ is nonproper.
Since $\cM$ is noncompact, Proposition~\ref{thm:obata-ferrand} implies that $(\cM, \cpi^\ast g)$ is conformally equivalent to flat Euclidean space $(\mathbb{R}^n, \delta)$.
\end{proof}

To prove our main result, we first show that if $(M,g)$ has positive Yamabe constant and $\pi_1(M)$ has infinite profinite completion, then for each $N\in\mathbb{N}$, there exists a finite regular covering $\Pi_N:M_N\to M$ such that $\#\mathcal{M}_{\rm geom}(M_N,[\Pi^*_Ng])\geqslant N$.
This also shows that the analytic moduli space of the conformal universal cover is infinite. 

\begin{proposition}
\label{prop:PDE-moduli}
Let $(M, g)$ be a compact $n$-dimensional Riemannian manifold with $n\geqslant 3$ such that $Y(M, [g]) > 0$.
Suppose that $\pi_1(M)$ has infinite profinite completion.
Then there exists an infinite tower of finite regular coverings
$\{\pi_j \colon M_j \to M_{j-1} \}_{j \in \bN}$
and a sequence of unit-volume Yamabe minimizers \( \{ \bar{g}_j \in [\Pi_j^\ast g] \}_{j\in\mathbb{N}} \), where \( \Pi_j := \pi_1 \circ \dotsb \circ \pi_j \colon M_j \to M \), such that for every \( j \in \mathbb{N}_0 \), the pullbacks
\(\{ (\Pi_j^\ell)^\ast \bar{g}_\ell\}_{\ell = 0}^j \subset [\Pi_j^\ast g]
\)
are pairwise nonhomothetic constant scalar curvature metrics, where $\Pi_j^\ell:=\pi_{\ell+1}\circ\cdots\circ \pi_j$ with the convention $\Pi^j_j={\rm Id}$. 
In particular, the conformal universal cover $(\cM,\cg)$ of $(M,g)$ satisfies  
\[\#\mathcal{M}_{\mathrm{PDE}}(\cM)\gtrsim\aleph_0.\]
\end{proposition}

\begin{proof}
 Let $\bar{g}_0 \in [g]$ be a unit-volume Yamabe minimizer.
 Denote the identity map by $\Pi_0 \colon M \to M$.

 Suppose that a finite regular covering $\Pi_j \colon M_j \to M$ and a unit-volume Yamabe minimizer $\bar{g}_j \in [\Pi_j^\ast g]$ are given.
 Since $\Pi_j$ is finite, $Y(M_j,[\Pi_j^\ast g])>0$.
 Let ${\rm v}_j > 0$ be such that
 \begin{equation}
     \label{eqn:pick-Vj}
     {\rm v}_j > \left( \frac{Y(S^n,[g_{S^n}])}{Y(M_j,[\Pi_j^\ast g])} \right)^{\frac{n}{2}}.
 \end{equation}
 Using Proposition~\ref{bettiol-piccione}, we can choose a finite regular covering $\pi_{j+1} \colon M_{j+1} \to M_j$ such that
 \begin{equation*}
     {\rm Vol}_{\pi_{j+1}^\ast \bar{g}_j}(M_{j+1}) > {\rm v}_j.
 \end{equation*}
 Set $\Pi_{j+1} := \Pi_j \circ \pi_{j+1} \colon M_{j+1} \to M$.
 By Proposition~\ref{prop:Yamabe-existence}, we may pick a unit-volume Yamabe minimizer $\bar{g}_{j+1}\in [\Pi_{j+1}^*g]$.
 Combining \eqref{eqn:pick-Vj} with the definition of $\pi_{\ell}$ yields
 \begin{equation*}
     \mathcal{A}(\pi_{j+1}^\ast\bar{g}_j) = Y(M_j,[\Pi_j^\ast g])\Vol_{\pi_{j+1}^\ast\bar{g}_j}(M_{j+1})^{\frac{2}{n}} > Y(S^n,[g_{S^n}]) \geq \mathcal{A}(\bar{g}_{j+1}) .
 \end{equation*}
 
 The above construction yields an infinite tower of finite regular coverings $\{ \pi_j \colon M_j \to M_{j-1} \}_{j\in\mathbb{N}}$ and a sequence of unit-volume Yamabe minimizers $\{ \bar{g}_j \in [\Pi_j^\ast g] \}_{j\in\mathbb{N}}$ such that $\mathcal{A}(\bar{g}_{j+1}) < \mathcal{A}(\pi_{j+1}^\ast\bar{g}_j)$ for all $j \in \mathbb{N}$.
 Since the metrics $\bar{g}_j$ all have constant scalar curvature, we deduce that if $0 \leq \ell \leq j$, then
 \begin{equation*}
    \mathcal{A}({(\Pi_j^{\ell-1})^\ast \bar{g}_{\ell-1}}) = \mathcal{A}({(\Pi_j^\ell)^\ast \pi_\ell^\ast \bar{g}_{\ell-1}}) =  \mathcal{A}({\pi_\ell^\ast \bar{g}_{\ell-1}})\left( \deg \Pi_j^\ell \right)^{\frac{2}{n}} > \mathcal{A}({\bar{g}_\ell})\left( \deg \Pi_j^\ell \right)^{\frac{2}{n}} = \mathcal{A}({(\Pi_j^\ell)^\ast \bar{g}_\ell}).
 \end{equation*}
 Therefore,
 \[
 \mathcal{A}({\bar{g}_j}) < \mathcal{A}({(\Pi_j^{j-1})^\ast \bar{g}_{j-1}}) < \dotsm < \mathcal{A}({(\Pi_j^0)^\ast \bar{g}_0}).
 \]
 Since each $M_j$ is compact, the scale and diffeomorphism invariance of the total scalar curvature implies that the metrics $\{(\Pi_j^\ell)^*\bar{g}_\ell\}_{\ell=0}^j$ are pairwise nonhomothetic.
 Hence, their conformal factors are distinct modulo a constant rescaling, a property that is preserved under pullback to the universal cover.
\end{proof}

We now turn to a deeper phenomenon. 
The previous result ensures that many distinct conformal factors exist that solve the Yamabe equation. 
However, these could, in principle, become equivalent upon lifting to the universal cover, where the isometry group can be much larger.
The following result shows that, under a natural geometric assumption, 
these conformal factors yield genuinely distinct metrics in the geometric moduli space.
\begin{theorem}
\label{thm:geom-moduli}
Let $(M^n,g)$ be a compact $n$-dimensional Riemannian manifold with \( n \geq 3 \) and such that \( Y(M,[g]) > 0 \). 
Suppose that \( \pi_1(M) \) has infinite profinite completion and the universal cover \( (\cM, \cpi^\ast g) \) is not conformally equivalent to Euclidean space. 
Then there exists an infinite tower of finite regular coverings
$\{\pi_j \colon M_j \to M_{j-1}\}_{j \in \bN}$ and a sequence of unit-volume Yamabe minimizers \( \{ g_j \in [\Pi_j^\ast g] \}_{j\in\mathbb{N}}\) such that their lifts \( \{ \cpi_j^\ast g_j \}_{j \in \mathbb{N}} \) to the universal cover are pairwise nonhomothetic. 
In particular, one has 
\[\#\mathcal{M}_{\mathrm{Geom}}(\cM)\gtrsim\aleph_0.\]
\end{theorem}

\begin{proof}
By Proposition~\ref{prop:PDE-moduli}, there exists an infinite tower \( \{ \pi_j \colon M_j \to M_{j-1} \}_{j\in\mathbb{N}} \) and metrics \( \{ g_j \}_{j\in \mathbb{N}} \) such that \( g_j \in [\Pi_j^\ast g] \), \( {\rm vol}_{g_j}(M_j) = 1 \), and \( R_{g_j} = Y_j(M_j, [g_j]) \). 
Now, let \( \cpi_j \colon \cM \to M_j \) be the universal coverings.
Suppose by contradiction that there is a subtower, still denoted $\{\Pi_j\}_{j\in\bN}$, such that for each $j,\ell\in\mathbb{N}$ there is a diffeomorphism $\Phi_{j,\ell}$ and a constant $c_{j,\ell}>0$ such that \(\Phi_{j,\ell}^*(\widetilde{\Phi}_{j}^*g_j)=c_{j,\ell}^2\widetilde{\Phi}_{j}^*g_j\). 
We deduce from Proposition~\ref{construction-lemma} that \( (\cM, \cpi^\ast g) \) is conformally equivalent to Euclidean space, contradicting the hypothesis. 
The conclusion readily follows.
\end{proof}

\subsection{Application to the singular Yamabe problem on the sphere}
Let us now apply Theorem~\ref{thm:geom-moduli} to the singular Yamabe problem. 

As a first step toward applying our general nonhomotheticity result to the case of singular Yamabe metrics on punctured spheres, we must verify that the initial compact manifold has a positive Yamabe constant and its conformal cover is noncompact and not conformally equivalent to Euclidean space.
This is elementary for the product of a round sphere with a hyperbolic manifold.

\begin{remark}
In the case $k=1$, by compact hyperbolic manifold $(\Sigma^1,g_{\Sigma^1})$, we mean $(S^1,g_{S^1})$.
\end{remark}

\begin{lemma}
\label{lem:positive-yamabe-nonround}
Let $k\in \mathbb{Z}_{\geqslant 0}$ and  \( n \geqslant 2k+3 \) and 
\((M, g) := ( S^{n-k-1} \times \Sigma^{k+1}, g_{S^{n-k-1}} \oplus g_\Sigma )
\) be a Riemannian product with \( \Sigma^{k+1}\subset H^{k+1} \) a compact hyperbolic manifold.
Then $Y(M, [g])>0$.
Moreover, the conformal universal cover \( (\cM, \pi^*g) \) is not the Euclidean space.
\end{lemma}

\begin{proof}
Observe that 
\[
R_g=R_{g_{S^{n-k-1}}}+R_{g_{H^{k+1}}}=(n-k-1)(n-k-2)-k(k+1)=(n-2k-2)(n-1)
\]
Since $n\geqslant 2k+3$, we see that $R_g>0$.
In addition, by compactness of $M$, we get $Y(M,[g])>0$.

Finally, the universal cover $\cM$ is diffeomorphic to $S^{n-k-1}\times \mathbb{R}^{k+1}$.
Since this is not contractible, the conformal universal cover is not the Euclidean space.
\end{proof}

We now prove our main result:
\begin{proof}[Proof of Theorem~\ref{thm:main}]
As in Lemma~\ref{lem:positive-yamabe-nonround}, let us consider the Riemannian product
\[
(M, g) := (S^{n-k-1} \times \Sigma^{k+1}, g_{S^{n-k-1}} \oplus g_\Sigma),
\]
where \( \Sigma^{k+1}\subset H^{k+1} \) is a compact hyperbolic manifold of dimension \( k\geqslant 1 \). 
Lemma~\ref{lem:positive-yamabe-nonround} implies that \( (M, g) \) has a positive Yamabe constant, but its conformal universal cover is not the Euclidean space.

By construction, the fundamental group of \( M \) is isomorphic to \( \pi_1(\Sigma) \), which is a cocompact lattice in \( \mathrm{SO}(k+1,1) \). In particular, \( \pi_1(M) \) is a finitely generated linear group over \( \mathbb{R} \). 
Proposition~\ref{lem:selberg-malcev-precise} then implies that $\pi_1(M)$ is residually finite. 
The conclusion now follows from Theorem~\ref{thm:geom-moduli}. 
\end{proof}

\section*{Declarations}

\subsection*{Funding}
This work was partially supported by Fundação de Amparo \`a Pesquisa do Estado de São Paulo (FAPESP), Conselho Nacional de Desenvolvimento Científico e Tecnológico (CNPq), Natural Sciences and Engineering Research Council of Canada (NSERC), Natural Sciences Foundation (NSF), Hong Kong Special Administrative Region General Research Fund (HKSAR CRF), and Simons Foundation. 
J.H.A. was supported by FAPESP \#2020/07566-3, \#2021/15139-0, and \#2023/15567-8 and CNPq \#409764/2023-0, \#443594/2023-6, \#441922/2023-6, and \#306014/2025-4. 
J.S.C. was partially supported by a Simons Foundation Collaboration Grant for Mathematicians and by the National Science Foundation under Award \#DMS-2505606.
P.P. was supported by FAPESP \#2022/16097-2, and CNPq \#313773/2021-1 and \#441922/2023-6. 
J.W. was supported by NSERC \#RGPIN-2018-03773 and  HKSAR-GRF General Research Grant \#14309824.

\subsection*{Conflict of interest}
The authors have no relevant financial or non-financial interests to disclose.

\subsection*{Data availability}
Data sharing is not applicable as no datasets were generated or analyzed during the study.

\subsection*{Ethics approval}
Not applicable.

\bibliography{references,bib}
\bibliographystyle{abbrv}

\end{document}